%% file: nilC-forPub.tex
\documentclass[12pt]{article}
\usepackage{amsmath}
\usepackage{amssymb}  
\usepackage{latexsym}
\usepackage{amsthm}
\usepackage{euscript}

\usepackage{txfonts} 
\usepackage{amsfonts}

\input{premathNEW}

\begin{document}      
 
\def\empty{\varnothing}
\def\emp{\varnothing}           
\def\mfrak{\mathfrak}

\renewcommand{\om}{\omega}
\reversemarginpar 	
 \pagestyle{plain}

 \def\mylabel#1{\label{#1}} 
	
\newcommand{\Z}[1]{\ensuremath{{\msf  Z ( #1)}}}
\renewcommand{\dim} {\ensuremath{{\mathsf {dim}\,}}}

 \title{\bf Fixed points of local actions of nilpotent Lie groups on
    surfaces}

 \author{\Large  Morris W. Hirsch\\
 Mathematics Department\\
University of Wisconsin at Madison\\
University of California at Berkeley}
\maketitle

\begin{abstract} 
Let $G$ be connected nilpotent Lie group acting locally on a real surface
$M$.  Let $\varphi$ be the local flow on $M$ induced by a
$1$-parameter subgroup.  Assume $K$ is a compact set of fixed
points of $\varphi$ and $U$ is a neighborhood of $K$ containing no
other fixed points.

{\em Theorem:} If the Dold fixed-point index of $\varphi_t|U$ is
nonzero for  sufficiently small $t>0$, then $\Fix G \cap
K\ne\varnothing$.
\end{abstract}

\tableofcontents

\section{Introduction}  \mylabel{sec:intro}
{\em Notation:}
$M$ is a manifold with tangent bundle $TM$ and boundary $\p M$.
 The set of vector fields on $M$ is $\vv (M)$, and the set of vector
 fields that are $C^1$ (continously differentiable) is $\vv^1
 (M)$. The zero set of $X\in  \vv (M)$ is $\Z X$.

$G$ always denotes
a connected Lie group with unit element $e$ and Lie algebra $\gg$.
We usually treat $\gg$ as the set of one-parameter subgroups
$X\co\RR\to G$, but we also exploit the linear structure on $\gg$
defined by its identification with the tangent space to $G$ at $e$.

$\RR$ denotes the vector space of real numbers, $\Rp=[0,\infty)$, $\R n$ is Euclidean
$n$-space.    $\ZZ$ is the group of
integers, $\NN:=\{0,1, \dots\}$ is the set of natural numbers, and  $\emp$ is  the empty set.     
The frontier of a subset $S$ of a topological space is denoted by $\fr S$. 
The fixed point set of  $f\co A \to B$ is 
$\Fix f:= \{a\in A\co f (a)=a\}$.  
Maps are continuous unless the
contrary is indicated. 

\smallskip
In 1885 Poincar\'e \cite{Poincare85} published a seminal result on
surface dynamics, extended to higher dimensions  by Hopf
\cite{Hopf25} in 1925: 
\begin{theorem*} [\sc Poincar\'{e}-Hopf]        
Every vector field on a closed manifold $M$ 
of nonzero Euler characteristic
vanishes at some point.
\end{theorem*}
A smooth vector field induces a flow $\Phi^X:=\{\Phi^X_t\}_{t\ge 0}$
  --- a continuous action of the group $\RR$ --- and the theorem implies
  $\Phi$ has a fixed point.  The same conclusion holds for semi-flows
  (continuous actions of $\Rp$) on a broad class of spaces including all
  compact polyhedra and topological manifolds, thanks to 
  Lefschetz's Fixed Point Theorem \cite{Lefschetz26}).  

In his pioneering 1964 paper, E. Lima \cite{Lima64} generalized the
Poincar\'e-Hopf Theorem to actions of connected abelian Lie groups on
compact surfaces, allowing nonempty boundaries.  This was extended to
nilpotent groups in 1986 by my former student
J. Plante \cite{Plante86} :

\begin{theorem*}    [\sc Plante] \mylabel{th:plante}
Every continuous  action of a connected nilpotent  Lie group on a compact
surface of nonzero Euler characteristic has a fixed point.
\end{theorem*}     

Our goal is a generalization of Plante's Theorem to local actions
(Section \ref{sec:local}) on all surfaces.  This necessitates
replacement of the assumption $\chi (M)\ne 0$.
The new hypothesis is based Dold's fixed point index $I(f)\in \ZZ$ for
maps $f\co U\to M$, where $U\subset M$ is open and $\Fix f$  (see Section
\ref{sec:index}).   

Let $\Phi:=\{\Phi_t\}_{t\in\RR}$ be a local flow on $M$.
 A {\em block for $\Phi$}
is a compact set
\[\textstyle
 K\subset \Fix \Phi :=\bigcap_{t\in\RR}\Fix{\Phi_t}
\] 
 having an {\em isolating neighborhood} $U$: a precompact open
 neighborhood of $K$ such that $K=\Fix \Phi \cap \ov U$.  When
 $\Phi:=\Phi^X$, the local flow generated by a vector field $X$ on $M$
 then $K\subset \Z X$, the {\em zero set} of $X$.

For sufficiently small $t >0$, the {\em index of $\Phi$ at $K$} is
well defined by the formula
\[\msf i_K (\Phi):=I (\Phi_t|U).\] 
When $\Phi$ comes from a vector field $X$ we define
\[
 \msf i_K (X)=\msf i (X, U) :=\msf i_K (X).
\]

In  Theorem \ref{th:MAIN} and its corollaries  $G$
is nilpotent, and a
local action $\Phi$ of $G$ on a surface $M$ is postulated (see Section
\ref{sec:local}).  Each $X\in \gg$ is a
one-parameter subgroup $\RR \to G$, inducing a local flow
$\Phi^X$ on $M$ whose fixed point set is denoted by $\Fix X$.  A block $K$
for $\Phi^X$ is an {\em $X$-block}. $K$  is {\em
  essential} if $\msf i_K (X)\ne 0$. 
\begin{theorem}    \mylabel{th:MAIN}
If $K$  is an essential $X$-block,  then $\Fix
G \cap K\ne\varnothing$.   
\end{theorem}

As distinct $X$-blocks are disjoint, we obtain:

\begin{corollary}               \mylabel{th:MAINcor1}
If $X$ has $n$ essential blocks, then $\Fix G$ has at least $n$
components.  \qed
\end{corollary}





Let $\Z \aa$ denote the set of common zeros of a set $\aa\subset\vv (M)$. 
\begin{corollary}               \mylabel{th:MAINcor2}
Let $\aa \subset \vv^1 (M)$ be a finite-dimensional linear subspace
tangent to $\p M$ and forming a nilpotent Lie algebra under the Lie
bracket operation.  If $K$ is an essential block of zeros for some
$X\in \gg$, then $\Z \aa \ne\varnothing$. \qed
\end{corollary}

The inspiration for Theorem \ref{th:MAIN} is a remarkable result of C. Bonatti
\cite{Bonatti92}:

 \begin{theorem*}[\sc {Bonatti}]
Assume $\dim M\le 4$, $\p M=\emp$, and $X$, $Y$ are commuting analytic
vector fields $M$. If $K$ is an essential block for the local flow
generated by 
$X$, then $\Z Y\cap K\ne\varnothing$.\footnote{
``The demonstration of this result involves a beautiful and quite
  difficult local study of the set of zeros of $X$, as an analytic
  $Y$-invariant set.  Of course, analyticity is an essential tool in
  this study, and the validity of this type of result in the smooth
  case remains an open— and apparently hard— question.'' ---P. Molino
  \cite{Molino93}}
  \end{theorem*}
This is one of the few fixed points theorem  for 
noncompact Lie groups ($\R 2$ in this case)  acting on 
manifolds that are not compact, or have dimensions $>2$, or have zero
Euler characteristic.  Another is Borel's Fixed Point Theorem, stated
below.
 
\smallskip
When $\p M=\empty$, our definition of the index of $X$ in $U$ extends
Bonatti's definition for vector fields, which runs as follows.  Let
$X$ be a $C^1$ vector field on $M$ generating the local flow $\psi$.
Let $U\subset M$ be an isolating neighborhood of an $X$-block $K$.
Then $\msf i_K(X)$ equals the intersection number of $X(U)$ and the
image of the trivial field on $U$ in the tangent bundle of $M$, which
is Bonatti's definition.  Equivalently: If a vector field $Y$ is
sufficiently close to $X$ and $\Z Y\cap U$ is finite, then $\msf
i_K(X)$ equals the the sum of the Poincar\'e-Hopf indices of the zeros
of $Y$ in $U$.  This sum, the {\em Poincar\'e-Hopf index} of $Y|U$, is
denoted by $\msf i_{PH} (Y|U)$.  See Proposition \ref{th:fundA}.

\subsection*{Discussion}
 Plante's theorem does not extend to Lie groups that are solvable, or
 even supersoluble,\footnote
{Supersoluble:  All eigenvalues in the adjoint representation are
  real. This implies solvable.}  because 
he   proved  \cite{Plante86}:
 
 \begin{itemize}      
\item Every compact surface  supports  a fixed-point free $C^\infty$ action by
  the  group  $H$ of real matrices
$ 
 \left[\begin{smallmatrix} 
   a & b\\ 
   0 & 1
  \end{smallmatrix}\right], \ (a >0)$. 
 \end{itemize}
 
On the other hand, sufficiently strong assumptions imply  fixed-point
theorem for several  natural classes of solvable group actions: 

 \smallskip
  $\bullet$ \ There is a fixed point in every algebraic action of a
 solvable, linear, irreducible algebraic group on a complete algebraic
 variety over an algebraically closed field (Borel \cite{Borel56,
   Borel69}; see also Humphreys \cite{Humphreys75}, Onishchik \&
 Vinberg \cite{Onish-Vinberg90}.)
This celebrated result needs no assumptions on dimensions, Euler
characteristics or compactness, and is valid for nonsmooth varieties.
 
\smallskip
$\bullet$ Borel's theorem extends to
holomorphic actions of connected solvable Lie groups on compact
K\"ahler manifolds $M$ with  $H^1 (M)=\{0\}$ ( A. Sommese \cite{Sommese73}).

\smallskip
  $\bullet$   $\Fix
G\ne\varnothing$ when $G$ is  supersoluble
 and acts analytically on a compact surface $M$ with $\chi (M) \ne 0$
(A. Weinstein and  M.\ W. Hirsch \cite{HW2000}).
But  this result fails for groups that arx
solvable but not supersoluble: The group of real matrices \,
$ 
 \left[\begin{smallmatrix} 
 a	 &  -b    & \  x \\ 
 b       &   a     & \ y \\
 0	 &  \ 0    & \ 1
  	\end{smallmatrix}\right], \ (a^2+b^2=1), \,
$
acts without fixed point on the 2-sphere of oriented lines through the origin
in $\R 3$.   And it  fails for $C^\infty$
actions, by Plante's Theorem. 

\smallskip
$\bullet$ \
 The conclusion of Bonatti's theorem holds for analytic vector fields $X, Y$ on a real or complex
$2$-manifold $M$ without boundary, satisfying $[X, Y]\wedge X=0$
 (Hirsch \cite{Hirsch2014}).   
 Two applications follow:

\smallskip
  $\bullet$ \ Let $\aa$ be a Lie algebra, perhaps infinite dimensional,
of analytic vector fields on a real or complex $2$-manifold with
empty boundary.  If $X\in \aa$ spans a one-dimensional ideal, $\Z \aa$
meets every essential $X$-block.

\smallskip
$\bullet$ \ Assume the center of $G$ has  positive
dimension and  $M$ is  a complex $2$-manifold with empty boundary
such that $\chi (M)\ne 0$.  Then every holomorphic action of $G$ on $M$ has a
fixed point.  

\smallskip
M.  Belliart \cite{Belliart97} classified the pairs $(G, M)$ where $G$
has a continuous fixed-point free action on a compact
surface $M$, relying on the classification of transitive surface
actions in Mostow \cite{Mostow50}.  In particular:

\smallskip
  $\bullet$ A solvable  $G$  acts without fixed point  on the
compact surface $M$ iff $G$ maps 
homomorphically onto $\msf {Aff}_+(\RR)$.
 
\smallskip
For related  work on the dynamics of Lie group actions see
\cite{BL96, ET79, Hirsch2010,  Hirsch2011, Horne65,  Hounie81}.

\paragraph{Open questions.} Is there an example of a connected nilpotent Lie group
  acting without fixed-point on a compact $n$-manifold, $n>2$, having
  nonzero Euler characteristic?  Does Bonatti's Theorem generalize
  to three or more vector fields,  or to
  manifolds of dimensions $ >4$?  

\section{Local actions}   \mylabel{sec:local}

For a map $f\co A\to B$ we adopt the convention that  notation such as 
$f(\xi)$ presupposes that $\xi$ is an element or subset of $A$.
 The domain $A$ of $f$ is denoted by $\mcal D f$ and the  range of $f$
 by  $\mcal Rf:=f(A)$. 

Let $g, \phi$ denote maps.  Regardless of the domains and ranges, the
composition $g\circ f$ is defined as the  map
(perhaps empty)
$g\circ f: x\mapsto g (f (x))$
whose  domain is $f^{-1}(\mcal D g)$.

The associative law holds for these compositions:  The maps 
$ (h\circ g)\circ f$ and $h\circ (g\circ f)$ have the same  domain  
\[D:= \{x\in \mcal D h\co h (x)\in \mcal D g, \quad g (h(x))\in
 \mcal D f\},
\]
and 
\[ x\in D\implies (h\circ g)(f(x)) =
h\circ ((g\circ f)(x)).
\]


A {\em local homeomorphism} on a topological space $Q$ is a
homeomorphism between open subsets of $Q$.  The set of these homeomorphisms
is denoted by $\msf {LH}(Q)$. 
\begin{definition}              \mylabel{th:defloc}
 A {\em local action} of the connected  Lie group $G$ on a manifold $M$ is a
 triple $(G, M,  \alpha)$
where 
$\alpha\co G \to
\msf {LH}(M)$
is a function 
having the following properties:
 
\begin{itemize}
\item
The set $\Omega_\alpha:=\big\{(g, p)\in G\times M\co p\in \mcal D
\alpha(g)\big\}$
is an open neighborhood of $\{e\} \times M$.

\item  The {\em evaluation map}
\[
\msf{ev_\alpha}\co \Omega_\alpha \to M,\quad (g, p)\mapsto \alpha(g) (p)
\]
is continuous. 

\item
 $\alpha(e)$ is the identity map of $M$.

\item  The maps $\alpha (fg)\circ \alpha (h)$,  $\alpha (f)\circ
  \alpha (gh)$ agree on the intersection of their domains, \
  ($f,g,h\in G$). 

\item
 $\alpha(e)$ is the identity map of $M$.

\item $\alpha (g^{-1})=\alpha (g)^{-1}$.
\end{itemize}
\end{definition}

$\alpha$  is a {\em global action} provided
$\Om_\alpha=G\times M$.  
 If $G$ is connected and simply
connected and $M$ is compact, every local action extends to a unique
global action.

In the rest of this section a local action $(G, M, \alpha)$ is
assumed.  We sometimes omit the notation ``$\alpha$'', writing $g$ for
$\alpha (g)$, $\mcal D g$ for $\mcal D \alpha(g)$, and so forth.

  A homomorphism $\phi\co H
\to G$ of Lie groups induces the local action $\alpha\circ\phi$ of $H$
on $M$ defined by
\[\alpha\circ \phi \co h\to \alpha (\phi(h)),
\]
called the {\em pullback} of $\alpha$ by $\phi$.  When $\phi$ is an
inclusion we set $\alpha|H:=\alpha\circ \phi$.  In this way $\alpha$
induces local actions of all Lie subgroups.

A {\em local flow} is a local action of the group $\RR$ of real
numbers. 
If $X\co \RR\to G$ is a
one-parameter subgroup, \, $\alpha \circ X$ is the local flow
characterized by 
\begin{equation}                \label{eq:pullback}
(\alpha\circ X) (t)\co p \mapsto \alpha(X(t))(p), \quad
( t\in \RR, \ p\in \mcal D \alpha (X(t)). 
\end{equation}

A Lie algebra homomorphism 
$X\mapsto \hat X$  from
$\gg$ to  $C^\infty$ vector fields on $M$ gives rise
to a local action $(G, M, \alpha)$ such that the 
maps $t\mapsto (\alpha\circ X (t))$ are the 
integral curves of
$\hat X$.
 (See
Palais \cite[{Th.\ }II.11]{Palais57}, also Varadarajan \cite
[{Th.\ }2.16.6]{Varadarajan76}).

A set $S\subset M$ is {\em invariant}  provided  
$g(S)$ is defined and in $S$ for all $g\in G$, or more equivalently:
\[S\subset \mcal D g \cap \mcal R g. \]  
The {\em orbit}  of $p\in M$ is the smallest invariant set
containing $p$.    
The {\em fixed-point set} of $\alpha$  is 
\[\textstyle
\Fix G:=\bigcap_{g\in G}\Fix g.
\]

\begin{proposition}             \mylabel{th:fixg}
$\Fix \gg = \Fix G$.
\end{proposition}
\begin{proof}  Equation (\ref{eq:pullback}) implies  $\Fix \gg \subset
  \Fix G$.  A neighborhood  of $e$ is covered by one-parameter
  subgroups, and it generates $G$ because $G$
  is connected.  This implies $\Fix G\subset \Fix \gg$.
\end{proof}
If $H\subset \gg$ is  a connected Lie subgroup
we set 
\[\Fix H=\Fix{\alpha|H}.
\]

\begin{proposition}             \mylabel{th:hg}
Let $G$ have a local action. 
\begin{description}

\item[(i)] If  $p\in \Fix h\cap\mcal Dg$,   then
$g(p)\in  \Fix{ghg^{-1}}.$

\item[(ii)] If $H\subset G$ is a connected normal Lie subgroup, $\Fix
  H$ is invariant under $G$.

\end{description}
\end{proposition}
\begin{proof} (i) is straightforward and implies (ii).
\end{proof}

\section{The fixed-point index}   \mylabel{sec:index}

The late A. Dold \cite {Dold65, Dold72} defined a fixed-point index
for a large class of maps having compact fixed-point
sets.  We use Dold'sindex  to define an index for blocks in local flows.
Dold's index $I (f)\in\ZZ$ is defined for data $f, V, M, \mcal S$ where

\begin{itemize}
\item $V$ is an open set in a topological space $\mcal S$,

\item $f\co V \to \mcal S$ is  continuous with $\Fix$ compact,

\item $V$ is a  {\em Euclidean neighborhood retract} (ENR): 
Some open set in a Euclidean space retracts onto  onto a homeomorph  of $V$.\footnote{The class of ENRs 
  includes metrizable topological manifolds and triangulable subsets
  of Euclidean spaces.} 
\end{itemize}
We will use the following  properties of $I(f)$:
\begin{description} 

\item{(D1)} \, $I(f)=I(f|V_0)$\, if $V_0\subset V$ is an open
  neighborhood of   $\Fix f$.  

\item{(D2)} \,  $I(f)=\begin{cases}
& 0 \ \mbox{if   $\Fix f =\empty$,}\\
& 1 \ \mbox{if   $f$ is constant.}
\end{cases}
$

\item{(D3)} \,  $I(f)=\sum_i I(f| V_i)$\, if $V$ is the 
  union of finitely many disjoint open sets $V_i$.

\item{(D4)} \,  $I(f_0)=I(f_1)$\, if there is a homotopy
  $f_t\co V\to \mcal S,\, (0\le t \le 1)$\, such that \,$\bigcup_t\Fix{f_t}$ is
  compact.
\end{description}
These correspond to (5.5.11), (5.5.12), (5.5.13) and
  (5.5.15) in  Chapter VII of Dold's book \cite{Dold72}.   

\begin{description} 
\item{(D5)} \,  Assume $\mcal S$ is a manifold, $f$ is $C^1$, \ $\Fix
  f$ is a singleton $p\in 
  M \,\verb=\=\, \p M$, and $\Det f'(p)\ne 0$.  Then
$I(f)= (-1)^\nu$, 
where $\nu$ is the number of distinct eigenvalues $\lam$ of $f'(p)$ such that $\lam
>1$.   (See  \cite[VII.5.17, Ex. 3]{Dold72}).)

\item{(D6)} \, Suppose   $\mcal S$ is compact, $V=\mcal S$, and $f$ is
  homotopic to the identity. Then $I
  (f)=\chi (\mcal S)$.  (See \cite[VII.6.22]{Dold72}.)  

\end{description} 


\subsection*{The index for local flows}
Let $\varphi:=
\{\varphi_t\}_{t\in \RR}$ be a local flow in a topological
space $\mcal S$.  

A compact set $K\subset\Fix \varphi$ is a {\em block} for $\varphi$,
or a {\em $\varphi$-block}, if it has an open, precompact ENR
neighborhood $V\subset \mcal S$ such that $\ov V\cap\Fix\varphi=K$.
Such a $V$ is said to be {\em isolating} for $\varphi$, and for
$(\varphi, K)$.  When $\varphi$ is smooth this language agrees with
the terminology for $X$-blocks in the Introduction, and $\Fix \varphi
= \Z X$.

It turns out that the fixed-point index $I (\varphi_t|V)$ for all 
sufficiently small $t>0$ depends only on $\varphi$ and $K$:

\begin{proposition}           \mylabel{th:tau}
If   $V$ is isolating for $\varphi$, 
there exists $\tau >0$ such that for all $t \in (0,\tau]$:

\begin{description}

\item[(a)] $\Fix {\varphi_t}\cap V$ is  compact,

\item[(b)] $I (\varphi_t|V)=I (\varphi_{\tau}|V)$.

\end{description}
\end{proposition}
\begin{proof}
If (a) fails there exist convergent sequences $\{t_k\}$ in  $[0,\infty)$,
and  $\{p_k\}$ in $V$,   such that
\[
 t_k\searrow 0, \quad
p_k\in \Fix{\varphi_{t_k}}\cap V, \quad p_k\to q\in \fr V.
\]       
Joint continuity of $(t,x)\mapsto \varphi_t (x)$ yields the contradiction
$q\in \Fix {\varphi} \cap \fr V$.  Assertion (b) is a
consequence of (a) and (D4).
\end{proof}

\begin{definition}              \mylabel{th:defeqindex}
Using  the notation of Proposition
\ref{th:tau} we  define the {\em  index}  of $\varphi$ in $V$, and at
$K$, as:
\[
 \msf i(\varphi,V)=\msf i_K (\varphi):= I(\varphi_\tau|V).
\]
$K$ and $V$ are  {\em essential} for $\varphi$ if $\msf i_K (V)\ne 0$.  
This implies $K\ne\varnothing$ by (D2). 
\end{definition}

 We say that 
 $\varphi$ is {\em smooth} and {\em generates $X\in \vv^1 (M)$}, provided
$\mcal S$ is a manifold $M$ and 
\[
 \left.\pd t\right|_{t=0} \varphi_t (p) = X_p,  \quad (p\in M).
\] 
This implies $X|\p M$ is tangent to $\p M$, and $\Z X=
\Fix{\varphi}$. 

Recall that $\msf i_{PH}(X)$ denotes the Poincar\'e-Hopf index for  vector
fields $X\in\vv(M)$ such that  $\p M=\emp$ and  $\Z X$ finite.  
\begin{proposition}             \mylabel{th:fundA}
Assume $\p M=\emp$.  Let  $\varphi$ be a smooth 
local flow on $M$ generating  $X\in\vv (M)$.
Suppose $V\subset M$ is isolating for $\varphi$.  Let $\{X^k\}$ be a
sequence in $\vv (M)$ converging to $X$ such that each set $\Z
{X^k}\cap \ov V$ is finite.   Then 
$\msf i (\varphi, V)= \msf i_{PH} (X^k|V)$ for sufficiently large $k$.
\end{proposition}
\begin{proof} Choose a sequence  $\{Y^k\} $ in $\vv^1 (M)$   with $\Z{Y^k}\cap\ov
  V$ finite, with $Y^k$ and so close to $X^k$ that $\msf i_{PH} (Y_k|V)= \msf
  i_{PH} (X^k|V)$ and $\lim_k Y_k= X$.   
Let $\psi^k$ denote the local
flow of $Y^k$.  There exists $\rho >0$ such that
\[\mbox{$\lim_{k\to\infty} \psi^k_t (x) =\varphi_t (x)$ \ uniformly for 
$(t,x) \in [0,\rho]\times \ov V$.}
\]  
The conclusion  follows by applying (D4) to $f_0:=\phi_t|V$ and
$f_1:=\psi^k_t|V$ for sufficiently small $t>0$. 
\end{proof}
 This result can be used to show that the Dold index and the
Bonatti index coincide in situations where both are defined. 

\smallskip
Now assume  $G$ has a  local action on $M$. 
Every $X\in \gg$ generates a local flow $\varphi^X$ on $M$.  A block
$K\subset \Fix{\varphi^X}$ is called 
 an $X$-block.   When $U\subset M$ is $U$ is isolating for $\alpha\circ X$
we say $U$ is isolating for $X$, and set
\[ \msf i (X, U)= \msf i_K (X)  := \msf i (\varphi^X, U).
\]
$K$ is  {\em essential for $X$} provided $\msf i_K (X)\ne 0$.
 
\begin{proposition}               \mylabel{th:stable}
Assume  $V\subset M$ is isolating for $X$. 
\begin{description}

\item[(a)]  The set  
\[\mathfrak N (X, V,\gg):= 
\big\{Y\in \gg \co \text{$V$ is isolating for $Y$ and \,$\msf i (Y, V)
  = \msf i (X, V)$}\big\}
\] 
is an open neighborhood of $X$ in   $\gg$. 

\item[(b)] If $\ov V$ is a compact invariant manifold,  
  $\msf i (Y, V)=\chi (\ov V)$ \,for all $Y\in \mathfrak N (X, V,\gg)$.  
\end{description}
\end{proposition}
\begin{proof}

{\em (a) } Compactness of $\fr V$ implies that the set 
$\mfrak N(\gg):=\{Y\in \vv (g) \co \Fix Y \cap \fr V =\emp$
 is an open
neighborhood of $X$, and  $V$ is isolating for every $Y\in \mfrak
N(\gg)$. 
If  $Y\in \mfrak N(\gg)$ 
is sufficiently close to $X$ and $0\le s\le 1$ then $Y_s:=  (1-s)X
+sY$ also lies in $\mfrak N (\gg)$, and therefore    
$\msf i (Y, V)
  = \msf i (X, V)$ by (D4).

{\em (b) } Follows from  (D6). 
\end{proof}

\section{Fixed point sets,  stabilizers and ideals}  \mylabel{sec:ideals}
As usual, $G$ denotes a connected  Lie group with Lie algebra
$\gg$.  When $G$ is nilpotent, its exponential map $\gg\to G$ is an
analytic diffeomorphism 
sending subalgebras onto closed subgroups, and ideals onto normal
subgroups. In some situations $\gg$ is more convenient than $G$
because it as a  natural linear structure.

 A local action  of $G$ on a surface $M$ is assumed.
Note that $\Fix G=\Fix \gg$ because $G$ is connected.
The {\em isotropy group of $p\in M$} is the subgroup $G_p\subset G$
  generated by
\[
 \{g\in G\co g(p)=p\}. 
\]
The {\em stabilizer of  $p\in M$} is the subalgebra $\gg_p\subset \gg$
generated by 
\[\big\{Y\in\gg\co p\in\Fix Y \big\}.
\]
The {\em stabilizer of $S\subset M$} is  $\gg_S:= \bigcap_{p\in S}\gg_p$.
Evidently a one-parameter subgroup $Y\co\RR\to G$ belongs to $\gg_p$
iff $Y (\RR)\subset G_p$. 

\begin{lemma}           \mylabel{th:covering}
If $\gg$ is nilpotent and $\dim (\gg) \ge 2$, 
every element of $\gg$  lies in an ideal of 
codimension one. 
\end{lemma}
\begin{proof}  If $\dim (\gg)= 2$ the conclusion is trivial because
  $\gg$ is abelian. 
 Assume inductively:  $\dim (\gg)=d\ge 3$ and
the lemma holds for Lie algebras of lower dimension. 
Let $Y\in \gg$ be arbitrary.  
Fix a $1$-dimensional central ideal $\jj$ and a surjective
Lie algebra homomorphism 
\[\pi\co \gg\to \gg/\jj.
\]  
By the inductive assumption $\pi (Y)$
belongs to a codimension-one ideal $\ff\subset  \gg/\jj$, whence $Y$
belongs the codimension-one ideal $\pi^{-1} (\ff) \subset \gg$.
\end{proof}

The following  simple result is  useful: 
\begin{proposition}             \mylabel{th:abh}
If $p\in M$ and  $\uu, \ww \subset \gg_p$ are linear subspaces such that $\uu
+\ww=\gg$, then $p\in \Fix \gg$.  \qed
\end{proposition}

The set $\mcal C (\gg)$  of codimension-one ideals  has a natural
structure as a projective variety in the real 
  projective space $\P{d-1},
   \,d=\dim (\gg)$, and is given the corresponding metrizable topology.

\begin{proposition}[\sc Plante  \cite{Plante86}]  \mylabel{th:V}
Assume $\gg$ is nilpotent. 
\begin{description}

\item[(i)]   Every component of $\mcal C(\gg)$ has positive dimension.

 \item[(ii)] Every codimension-one subalgebra of $\gg$  is an ideal. 

\item[(iii)]  If   $O\subset M$ is  a one-dimensional orbit of the
  local action, then  
\[
  p\in  O\implies  \gg_p=\gg_O \in\mcal C(\gg).  
\]

\end{description}
\end{proposition}

 \subsection*{Minimal sets} \mylabel{sec:minimal}

 A {\em minimal set} (for the local action of $G$ on $M$) is a nonempty compact
invariant set containing no 
smaller such set.  Compact orbits are minimal sets; all other minimal
sets are {\em exceptional}. 

An orbit homeomorpic to the unit circle $\S 1$ is a {\em circle orbit}.  
A circle orbit is  {\em isolated} if it has a neighborhood
containing no other circle orbit; otherwise it is  {\em nonisolated}.

The following proposition is adapted from Plante
\cite{Plante86}:

\begin{proposition}             \mylabel{th:plante2}
Let $M_1\subset M$ be a compact surface.
\begin{description}

\item[(i)]   The number of exceptional minimal sets in $M_1$
is at most half the
genus of $M_1$.

\item[(ii)] The union of the minimal sets in $M_1$ is compact.

\item[(iii)] The union of  the nonisolated circle orbits in $M_1$ is compact.

\item[(iv)] 
If $C\subset M$  is a nonisolated circle orbit, every neighborhood of
$C$  contains a compact invariant
surface $Q$ such that: 
\begin{itemize}

\item each component $P$ of $Q$ is either an annulus or a M\"obius
  band,

\item $\Lam\,\verb=\=\,P$ contains at most finitely many minimal
  sets. 
\end{itemize}
\end{description}
\end{proposition}
\begin{proof}  (i) is a generalization of 
  \cite[Lemma 2.3]{Plante86}, with the same proof.  Slight revisions
  of arguments on \cite[page 155]{Plante86} prove the other assertions.
\end{proof}

\section{Proof of Theorem \ref{th:MAIN}}
  \mylabel{sec:proofs}  
Recall that $G$ is a connected nilpotent Lie group with a local action
on a surface $M$, the Lie algebra of $G$ is $\gg$, and $K\subset M$ is
an essential block of fixed points for the induced local flow of a
one-parameter subgroup $X\co \RR \to G$.  

The theorem states that $\Fix G\cap K\ne\varnothing$, which is
trivial if $\dim G \le 1$.  Assume inductively: $\dim G >1$ and the
conclusion holds for groups of lower dimension.

Every neighborhood of $K$ in $M$ contains an isolating neighborhood $U$ for
$(X, K)$ such  that $\ov
U$ is a compact surface. 
It suffices to prove for all such $U$ that
\begin{equation}                \label{eq:MAIN*}
 \Fix G \cap U\ne\varnothing.
\end{equation}

\begin{lemma}           \mylabel{th:wm}
If $U$ contains only 
finitely many minimal sets,  Equation (\ref{eq:MAIN*}) holds.
\end{lemma}
\begin{proof}  
Since $\dim (\gg)>1$ and $\gg$ is covered by codimension-one ideals
(Lemma \ref{th:covering}), $\gg$ contains a set
$\{Y_k\}_{K\in \NN}$ converging to $X$, and a sequence $\{\hh_k\}$ of pairwise 
distinct, codimension-one ideals, such that: $Y_k\in\hh_k$, $U$ is
isolating for $Y_k$, and $\msf i (Y, U)\ne 0$ (Proposition
\ref{th:stable}).
As the set $K_k := \Fix {\hh_k}\cap U $ is compact, nonempty by the
induction hypothesis, and invariant by Proposition \ref{th:hg}(ii),
there is a minimal set $L_k\subset K_k\subset U$.  The
hypothesis of the Lemma implies there exist indices $i, j$ such that
$\hh_i\ne \hh_j$ and $L_i= L_j$ .  Equation (\ref{eq:MAIN*}) now
follows from Proposition \ref{th:abh}.
\end{proof}

In verifying Equation (\ref{eq:MAIN*}) we can assume $ U$ contains
infinitely many minimal sets, thanks to Lemma \ref{th:wm}.  Setting
$M_1:=\ov U$ in Proposition \ref{th:plante2}, we see that all but
finitely many of these are circle orbits, and there is a nonempty,
compact, invariant surface $P\subset \ov U$ such that:
\begin{equation}                \label{eq:pm1a}
\chi (P)=0 \ \mbox{and $\ov U \setminus P$ contains only finitely
  many minimal sets.}  
\end{equation}
If $\Fix G\cap P\ne\varnothing$,   Equation
(\ref{eq:MAIN*}) holds and the proof is complete.
Henceforth  assume:
\begin{equation}                \label{eq:FGP}
\Fix G\cap P=\emp
\end{equation}

 \begin{lemma}           \mylabel{th:z}
There exists  $Z\in \gg$ such that:

\begin{description}
\item[(a)]  $U$ is isolating for $Z$, 

\item[(b)] $\msf i (Z, U)=\msf i (X, U)\ne 0$,

\item[(c)] $\Fix Z \cap \p P=\emp$. 

\end{description}
\end{lemma}
\begin{proof}
 Each of the finitely many components $C_i\subset \p P$ contains no
 fixed point by (\ref{eq:FGP}), hence it is a circle orbit.
Proposition \ref{th:V}(iii)   
  shows that property (c) holds for all $Z$ in
 the dense open set $\gg\,\verb=\=\, \bigcup_i\gg_{C_i}$, while (a)
 and (b) hold if  $Z$  is in the nonempty open set $\mfrak N (X,
 U,\gg)$ (see Proposition \ref{th:stable}).  Thus the Lemma is satisfied by
 all  $Z$ in the nonempty set $\mfrak N (X, U,\gg)
 \,\verb=\=\,\bigcup\gg_{C_i}$.
\end{proof}

Fix  $Z$  as in Lemma  \ref{th:z}.  Then
\begin{equation}                \label{eq:zun0}
\msf i (Z, U)\ne 0.
\end{equation}
Since  $\Fix Z \cap\p P=\emp$ by Lemma \ref{th:z}(c), the sets
$U\,\verb=\=\,P$ and $P\,\verb=\=\,\p P$ are isolating for
$Z$.  Therefore 
\begin{equation}                \label{eq:Z1A}
 \msf i (Z, U) =\msf i (Z, U\verb=\= P) +  \msf i (Z, P\,\verb=\=\, \p P). 
\end{equation}
by  (D3) in Section \ref{sec:index}.
Now $\msf i (Z, P\,\verb=\=\,\p P)=0$, by 
Theorem \ref{th:stable}(b) with $V:=P\,\verb=\=\,\p P$.  Consequently 
\begin{equation}                \label{eq:Z2}
\msf i (Z, U\,\verb=\=\,    \p P) \ne 0
\end{equation}
by (\ref{eq:zun0}) and (\ref{eq:Z1A}).  Equations (\ref{eq:pm1a}) and
(\ref{eq:Z2}) show that $U\,\verb=\=\, P$ contains only finitely many
minimal sets.  Therefore Lemma \ref{th:wm} yields
 \[
 \Fix G \cap  (U\,\verb=\=\, P) \ne\varnothing,
\]
implying  (\ref{eq:MAIN*}).  This completes the proof of Theorem \ref{th:MAIN}.
\qed


\end{document}

%% file: premathNEW.tex
\newcommand{\msf}[1]{\mathsf {#1}}

\newcommand{\mcal}[1]{{\mathcal {#1}}} 

\newtheorem{theorem}{Theorem}  
\newtheorem{lemma}[theorem]{Lemma}
\newtheorem{corollary}[theorem]{Corollary}
\newtheorem{proposition}[theorem]{Proposition}

\theoremstyle{definition}
\theoremstyle{definition}\newtheorem{definition}[theorem]{Definition}
\theoremstyle{definition}
\theoremstyle{definition}

\theoremstyle{definition}
\theoremstyle{definition}

\theoremstyle{plain}\newtheorem*{theorem*}{Theorem}
\theoremstyle{plain}\newtheorem*{corollary*}{Corollary}
\theoremstyle{remark}\newtheorem*{remark*}{Remark}
\theoremstyle{remark}\newtheorem*{remarks*}{Remarks}
\theoremstyle{definition}\newtheorem*{conjecture*}{Conjecture}
\theoremstyle{definition}\newtheorem*{definition*}{Definition}
\theoremstyle{definition}\newtheorem*{example*}{Example}

\theoremstyle{definition}\newtheorem*{question*}{Question}
\theoremstyle{definition}\newtheorem*{questions*}{Questions}

\def\qed{\ifhmode\unskip\nobreak\fi\ifmmode\ifinner\else\hskip5pt\fi\fi
 \hfill\hbox{\hskip5pt\vrule width4pt height6pt depth1.5pt\hskip1pt}}
 
\newcommand{\ov}[1]{\ensuremath{\overline{#1}}}

\newcommand{\Det}{\ensuremath{\operatorname{\mathsf{Det}}}}

\newcommand{\RR}{\ensuremath{{\mathbb R}}}     
\newcommand{\R}[1]{\ensuremath{{\mathbb R}^{#1}}} 

\newcommand{\Rp}{\ensuremath{{\mathbb R}_+}}   

\newcommand{\ZZ}{\ensuremath{{\mathbb Z}}}	   
\newcommand{\NN}{\ensuremath{{\mathbb N}}}	   
\renewcommand{\S}[1]{\ensuremath{{\bf S}^{#1}}} 
\renewcommand{\P}[1]{\ensuremath{{\bf P}^{#1}}} 

\newcommand{\om}[1]{\ensuremath{\omega(#1)}}

\newcommand {\pd}[1] {\ensuremath{\frac{\partial}{\partial #1}}}

\def\d#1dt{\frac{d#1}{dt}}    
\newcommand{\Om}{\ensuremath{\Omega}}

\newcommand{\lam}{\ensuremath{\lambda}}
\newcommand{\Lam}{\ensuremath{\Lambda}}

\newcommand{\Fix}[1]{\ensuremath{\operatorname{\mathsf {Fix}}(#1)}}
\newcommand{\co}{\colon\thinspace} 

\def\empty{\varnothing}

\newcommand{\fr}[1]{\ensuremath{\operatorname{\mathsf{Fr}}(#1)}}

\renewcommand{\gg}{\ensuremath{\mathfrak g}}
\renewcommand{\aa}{\ensuremath{\mathfrak a}}

\newcommand{\ff}{\ensuremath{\mathfrak f}}
\newcommand{\hh}{\ensuremath{\mathfrak h}}

\newcommand{\jj}{\ensuremath{\mathfrak j}}

\newcommand{\uu}{\ensuremath{\mathfrak u}}
\newcommand{\vv}{\ensuremath{\mathfrak v}}
\newcommand{\ww}{\ensuremath{\mathfrak w}}

\newcommand{\p}{\ensuremath{\partial}}